\documentclass[reqno,12pt]{amsart}
\usepackage{amsmath, amssymb}
\usepackage{amsthm, amsbsy}
\usepackage{graphicx}
\usepackage{float}
\usepackage{stackrel}

\theoremstyle{plain}
\newtheorem{theorem}{\indent\bf Theorem}[section]

\newtheorem{corollary}[theorem]{\indent\bf Corollary}

\theoremstyle{definition}
\newtheorem{definition}[theorem]{\indent\bf Definition}

\begin{document}
\title[Irreducibility criteria]{Apollonius circles and irreducibility criteria\\for 
polynomials}    
\author[A.I. Bonciocat]{Anca Iuliana Bonciocat}
\address{Simion Stoilow Institute of Mathematics of the Romanian 
Academy, Research Unit nr. 3,
P.O. Box 1-764, Bucharest 014700, Romania}
\email{Anca.Bonciocat@imar.ro} 
\author[N.C. Bonciocat]{Nicolae Ciprian Bonciocat}
\address{Simion Stoilow Institute of Mathematics of the Romanian 
Academy, Research Unit nr. 7,
P.O. Box 1-764, Bucharest 014700, Romania}
\email{Nicolae.Bonciocat@imar.ro}
\author[Y. Bugeaud]{Yann Bugeaud}
\address{Universit\'{e} de Strasbourg, Math\'{e}matiques, 7, 
rue Ren\'{e} Descartes, 67084 Strasbourg Cedex, France}
\email{yann.bugeaud@math.unistra.fr}
\author[M. Cipu]{Mihai Cipu}
\address{Simion Stoilow Institute of Mathematics of the Romanian 
Academy, Research Unit nr. 7,
P.O. Box 1-764, Bucharest 014700, Romania}
\email{Mihai.Cipu@imar.ro}

\dedicatory{Dedicated to our friend, Professor Maurice Mignotte}


\begin{abstract}
We prove the irreducibility of integer polynomials $f(X)$ whose roots lie 
inside an Apollonius circle associated to two points on the real axis with 
integer abscisae $a$ and $b$, with ratio of the distances to these points 
depending on the canonical decomposition of $f(a)$ and $f(b)$. In particular, 
we obtain irreducibility criteria for the case where $f(a)$ and $f(b)$ have 
few prime factors, and $f$ is either an Enestr\"om-Kakeya polynomial, or has 
a large leading coefficient. Analogous results are also provided for 
multivariate polynomials over arbitrary fields, in a non-Archimedean setting.
\end{abstract}
\maketitle

\section{Introduction} \label{se1}

One of the methods to study the irreducibility of polynomials is to use 
information on the values that they take at some specified integer 
arguments. A famous result of P\'{o}lya \cite{Polya} considers only the 
magnitude of  the absolute values that a polynomial takes, with disregard 
to their canonical decomposition:
\medskip

{\bf Theorem 1.}\ {\em If for $n$ integral values of $x$, the integral 
polynomial $f(x)$ of degree $n$ has values 
which are different from zero, and in absolute value less than
\[
\frac{\lceil n/2\rceil! }{2^{\lceil n/2\rceil }},
\]
then $f(x)$ is irreducible over $\mathbb{Q}$.
}
\medskip
 
Since 1919 this result was generalized in many different ways, of which 
we only mention here two recent ones, corresponding to the setting where 
the coefficients belong to the ring of integers of an arbitrary imaginary 
quadratic number field \cite{GHT}, and to the multivariate case over an 
arbitrary field \cite{BBCM1}.

Other irreducibility criteria in the literature rely heavily on the 
canonical decomposition of the value that a given polynomial takes at a 
single, specified integral argument. The most interesting results of this 
kind take benefit of the existence in this canonical decomposition of a 
suitable prime divisor, or prime power divisor. For instance, in 
\cite{PolyaSzego} P\'{o}lya and Szeg\"{o} give the following nice 
irreducibility criterion of A. Cohn:\medskip

{\bf Theorem 2.} \ \emph{If a prime $p$ is expressed in the decimal
system as 
\[
p=\sum\limits_{i=0}^{n}a_{i}10^{i},\quad0\leq a_{i}\leq9,
\]
then the polynomial $\sum_{i=0}^{n}a_{i}X^{i}$ is irreducible in $\mathbb{Z}
[X]$.}\medskip

Brillhart, Filaseta and Odlyzko \cite{Brillhart} extended this result to an
arbitrary base $b$: \medskip

{\bf Theorem 3.} \ \emph{If a prime $p$ is expressed in the number system
with base $b\geq2$ as 
\[
p=\sum\limits_{i=0}^{n}a_{i}b^{i},\quad0\leq a_{i}\leq b-1,
\]
then the polynomial $\sum_{i=0}^{n}a_{i}X^{i}$ is irreducible in $\mathbb{Z}
[X]$.}\medskip

Filaseta \cite{Filaseta1} obtained another generalization
of Cohn's theorem by replacing the prime $p$ by a composite number $pq$ with $q<b$
:\medskip

{\bf Theorem 4.} \ \emph{Let $p$ be a prime number, $q$ and $b$ positive
integers, $b\geq2$, $q<b$, and suppose that $pq$ is expressed in the number
system with base $b$ as 
\[
pq=\sum\limits_{i=0}^{n}a_{i}b^{i},\quad0\leq a_{i}\leq b-1.
\]
Then the polynomial $\sum_{i=0}^{n}a_{i}X^{i}$ is irreducible over the
rationals.}\medskip

Cohn's irreducibility criterion was also generalized in \cite{Brillhart} and 
\cite{Filaseta2} by permitting the coefficients of $f$ to be different from
digits. For instance, the following irreducibility criterion for polynomials
with non-negative coefficients was proved in \cite{Filaseta2}.\medskip

{\bf Theorem 5.} \ \emph{Let $f(X)=\sum_{i=0}^{n}a_{i}X^{i}$ be such that 
$f(10)$ is a prime. If the $a_{i}$'s satisfy $0\leq a_{i}\leq a_{n}10^{30}$
for each $i=0,1,\dots,n-1$, then $f(X)$ is irreducible.} \medskip

Cole, Dunn, and Filaseta produced in~\cite{CDF} sharp bounds $M(b)$ depending
on an integer $b\in [3,20]$ such that if each coefficient of a polynomial $f$ 
with non-negative integer coefficients is at most $M(b)$ and $f(b)$ is prime, 
then $f$ is irreducible.

Some classical related results relying on
the canonical decomposition of the value that a polynomial takes at some integral argument may be also found in the works of St\"ackel \cite{Stackel}, Ore \cite{Ore}, Weisner \cite{Weisner} and Dorwart \cite{Dorwart}. 
For an unifying approach that uses the concept of admissible triples to study irreducibility of polynomials, we refer the reader to \cite{Guersenzvaig}. Along with a simultaneous generalization of some classical irreducibility criteria, one may also find in \cite{Guersenzvaig} upper bounds for the total number of irreducible factors (counting multiplicities) for some classes of integer polynomials (see also \cite{GS} for problems related to the study of roots multiplicities and square free factorization). 
For further related results and some elegant connections between prime numbers and 
irreducible polynomials, the reader is referred to \cite{RamMurty}, \cite{Girstmair} and \cite{BDN3}, for instance. 

Another method to obtain irreducible polynomials is to write prime numbers 
or prime powers as a sum of integers of arbitrary sign, of which one has 
sufficiently large modulus, and to use these integers as coefficients of 
our polynomials. In this respect we refer the reader to \cite{Bonciocat1} 
and \cite{Bonciocat2}, where several irreducibility criteria for polynomials 
that take a prime value or a prime power value and have a coefficient of 
sufficiently large modulus have been obtained. Two such irreducibility 
criteria are given by the following results:
\medskip

{\bf Theorem 6.} \ \emph{If we write a prime number as a sum of integers $
a_{0},\dots,a_{n}$, with $a_{0}a_{n}\neq 0$ and $|a_{0}|>
\sum_{i=1}^{n}|a_{i}|2^{i}$, then the polynomial $\sum_{i=0}^{n}a_{i}X^{i}$
is irreducible over $\mathbb{Q}$.} \medskip

{\bf Theorem 7.} \ \emph{If we write a prime power $p^s$, $s\geq 2$, as a sum 
of integers $a_{0},\dots,a_{n}$ with $a_{0}a_{n}\neq 0$,
$|a_{0}|>\sum_{i=1}^{n}|a_{i}|2^{i}$, and $a_1+2a_2+\cdots +na_n$ not divisible 
by $p$, then the
polynomial $\sum_{i=0}^{n}a_{i}X^{i}$ is irreducible over $\mathbb{Q}$.}
\medskip

Other recent results where prime numbers play a central role in testing 
irreducibility refer to linear combinations of relatively prime polynomials 
\cite{CAV}, \cite{CVZ}, \cite{BBCM2}, and to
compositions of polynomials \cite{Guersenzvaig2} and
\cite{BBCM4}.
Counterparts of such results for the multivariate case may be found in \cite{CVZ2}, \cite{BBCM3}, \cite{BZ1}, and \cite{BZ2}. For some recent fundamental results on reduction, specialization and composition of polynomials in connection with Hilbert Irreducibility Theorem, Bertini-Noether Theorem and Schinzel Hypothesis, we refer the reader to \cite{Debes5}, \cite{BSE}, \cite{BDN1}, \cite{BDN2} and \cite{LBS}.

The aim of this paper is to provide irreducibility criteria that depend on the information on the canonical decomposition of the values that a polynomial $f$ takes at two integer arguments, 
by also using information on the location of their roots, and then to obtain similar results in the multivariate case over an arbitrary field. As we shall see, to obtain sharper irreducibility conditions we will also make use of information on the derivative of $f$, or on the partial derivatives of $f$ in the multivariate case. First of all, let us note that if a polynomial $f(X)\in \mathbb{Z}[X]$ factors as 
$f(X)=g(X)h(X)$ with $g(X),h(X)\in \mathbb{Z}[X]$ and $\deg g\geq 1$, $\deg h\geq 1$, 
then if we fix an arbitrarily chosen integer $a$ with $f(a)\neq 0$, the integers 
$g(a)$ and $h(a)$ are not some arbitrary divisors of $f(a)$, as they must also 
satisfy the equality $f'(a)=g'(a)h(a)+g(a)h'(a)$. 
It implies that the greatest common divisor of $g(a)$ and $h(a)$ divides $f(a)$ and $f'(a)$. 
This suggests the use of the following definition.

\begin{definition}\label{admissible}
Let $f$ be a non-constant polynomial with integer coefficients, and let $a$ be an integer with $f(a)\neq 0$. 
We 
say that an integer $d$ is an {\it admissible divisor} of $f(a)$ if $d\mid f(a)$ and 
\begin{equation}\label{adm2}
\gcd \left( d,\frac{f(a)}{d}\right) \mid \gcd (f(a),f'(a)),
\end{equation}
and we shall denote by $\mathcal{D}_{ad}(f(a))$ the set of all admissible divisors of $f(a)$. 
We say that an integer $d$ is a {\it unitary divisor} of $f(a)$ if $d$ is coprime with $f(a)/d$. 
We denote by $\mathcal{D}_{u}(f(a))$ the set of unitary divisors of $f(a)$.  
\end{definition}

We note that condition (\ref{adm2}) is symmetric in $d$ and its complementary divisor $f(a)/d$, 
and that if $\gcd (f(a),f'(a))=1$, then $\mathcal{D}_{ad}(f(a))$ 
reduces to the set $\mathcal{D}_{u}(f(a))$.

The first result that we will prove relies on information on the 
admissible divisors of $f(a)$ and $f(b)$ for two integers $a,b$. 
Rather surprisingly, the study of the irreducibility of $f$ can be connected with the location of the roots of $f$ inside an Apollonius circle associated 
to the points on the real axis with integer abscisae $a$ and $b$, 
and ratio of the distances to these two points expressed only in terms 
of the admissible divisors of $f(a)$ and $f(b)$. We recall the famous 
result of Apollonius, stating that the set of points $P$ in the plane such that the ratio of  distances from $P$  to two fixed points $A$ and $B$ equals some specified  $k$ is a circle (see  Figure 1),  which may degenerate to a point (for $k\to 0$ or $k\to \infty$) or to a line (for $k\to 1$). 
\begin{center}
\hspace{2.6cm}
\setlength{\unitlength}{8mm}
\begin{picture}(12,5.5)
\linethickness{0.15mm}

\put(-1,2){\vector(1,0){9}}
\put(0,0.25){\vector(0,1){4.75}}

\thicklines
\linethickness{0.1mm}
\put(4,1.3){\line(0,1){0.15}}
\put(4,1.5){\line(0,1){0.15}}
\put(4,1.7){\line(0,1){0.15}}
\put(4,1.9){\line(0,1){0.15}}
\put(4,2.1){\line(0,1){0.15}}
\put(4,2.3){\line(0,1){0.15}}
\put(4,2.5){\line(0,1){0.15}}
\put(4,2.7){\line(0,1){0.15}}
\put(4,2.9){\line(0,1){0.15}}
\put(4,3.1){\line(0,1){0.15}}
\put(4,3.3){\line(0,1){0.15}}
\put(4,3.5){\line(0,1){0.15}}
\put(4,3.7){\line(0,1){0.15}}

{\tiny 
\put(3.25,2.75){$P$}
\put(1.65,3.15){$k>1$}
\put(5.55,3.15){$k<1$}
\put(3.55,4.15){$k=1$}
\put(7.55,3.75){$d(P,B)=k\cdot d(P,A)$}

\put(2.35,1.55){$(a,0)$}
\put(4.75,1.55){$(b,0)$}

\put(2.7,2.2){$A$}
\put(4.97,2.2){$B$}

\put(3.35,0.8){$x=\frac{a+b}{2}$}

}

\put(4,2){\circle{0.08}}

\put(3,2){\circle{0.08}}
\put(3,2){\circle{0.12}}

\put(4.95,2){\circle{0.08}}
\put(4.95,2){\circle{0.12}}

\put(2.666,2){\circle{1.82}}
\put(5.333,2){\circle{1.82}}

\put(3.36,2.53){\circle{0.08}}
\put(3.36,2.53){\circle{0.12}}

\put(3,2){\line(2,3){0.35}}  

\put(3.36,2.53){\line(3,-1){1.6}}  

{\small 
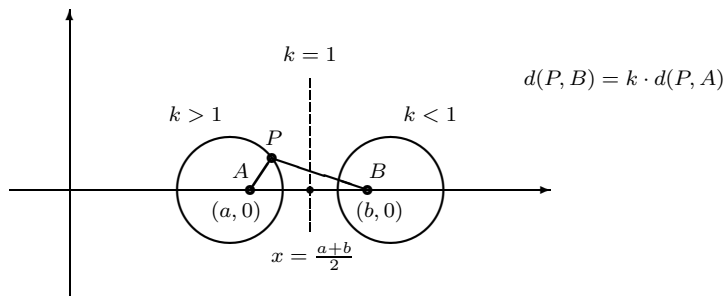
\put(-5.6,-0.65){{\bf Figure 1.}\ The Apollonius circles with respect to a pair of points in the plane}
}
\end{picture}
\end{center}
\bigskip

More precisely, given two points $A=(a,0)$ and $B=(b,0)$ 
and $k>0$,             
the set of points $P=(x,y)$ with $d(P,B)=k\cdot d(P,A)$ is the 
Apollonius circle ${\rm Ap}(a,b,k)$ given by the equation
\begin{equation}\label{Apollonius1}
\left( x-a+\frac{b-a}{k^2-1}\right)^2+y^2=k^2\left( \frac{b-a}{k^2-1}\right) ^2.
\end{equation}
Our first result that establishes a connection between Apollonius 
circles and irreducibility testing is the following.
\begin{theorem}\label{thm0}
Let $f(X)=a_{0}+a_{1}X+\cdots +a_{n}X^{n}$ be a polynomial with integer 
coefficients, and assume that for two integers $a,b$ we have $0<|f(a)|<|f(b)|$. Let
\begin{equation}\label{primulq}
q=\max \left\{ \frac{d_2}{d_1}\leq \sqrt{\frac{|f(b)|}{|f(a)|}}: 
d_1\in \mathcal{D}_{ad}(f(a)), \ d_2\in \mathcal{D}_{ad}(f(b)) \right\} .
\end{equation}

\emph{i)} If $q>1$ and all the roots of $f$ lie inside the Apollonius  
circle ${\rm Ap}(a,b,q)$, then $f$ is irreducible over $\mathbb{Q}$.

\emph{ii)} If $q>1$, all the roots of $f$ lie inside the Apollonius circle ${\rm Ap}(a,b,\sqrt{q})$
and $f$ has no rational roots, then $f$ is irreducible over $\mathbb{Q}$.

\emph{iii)} Assume that $q=1$. If $b>a$ and all the roots of $f$ lie in 
the half-plane $x<\frac{a+b}{2}$, or if $a>b$ and all the roots of $f$ lie 
in the half-plane $x>\frac{a+b}{2}$, then $f$ is irreducible over $\mathbb{Q}$.
\end{theorem}

As we shall see in the sequel, in general it is desirable to work with values
of $q$ in the statement of Theorem \ref{thm0} as small as possible, in order 
to relax the constraints on the two integers $a$ and $b$ that we use. For 
instance, if $b>a$ and we can prove that $q=1$ (which is the minimum possible 
value of $q$), by imposing the condition that $f(X+\frac{a+b}{2})$ is a Hurwitz stable
polynomial, so that all the roots of $f$ lie in the half-plane $x<\frac{a+b}{2}$,       
then by Theorem \ref{thm0} iii) we may conclude that $f$ is irreducible over $\mathbb{Q}$.  As known, a necessary and sufficient condition 
for a polynomial to be Hurwitz stable is that it passes the Routh--Hurwitz test.

In some applications, instead of testing the conditions in Theorem \ref{thm0}, 
it might be more convenient to consider the maximum of the absolute values 
of the roots of $f$, as follows.

\begin{theorem}\label{thm1} 
Let $f(X)=a_{0}+a_{1}X+\cdots +a_{n}X^{n}$ be a polynomial with integer 
coefficients, $M$ the maximum of the absolute values of its roots,
and assume that for two integers $a,b$ we have $0<|f(a)|<|f(b)|$. Let 
$q$ be given by {\em (\ref{primulq})}.

\emph{i)} \ \thinspace \thinspace If $|b|>q|a|+(1+q)M$, then $f$ is 
irreducible over $\mathbb{Q}$. 

\emph{ii)} \thinspace \thinspace If $|b|>\sqrt{q}|a|+(1+\sqrt{q})M$ and 
$f$ has no rational roots, then $f$ is irreducible over $\mathbb{Q}$. 

\emph{iii)} If $q=1$, $a^2<b^2$ and $M<\frac{|a+b|}{2}$, then $f$ 
is irreducible over $\mathbb{Q}$.
\end{theorem}

We note that one can obtain slightly weaker results by allowing $d_1$ 
and $d_2$ in the definition of $q$ in the statement of Theorem \ref{thm0} 
to be arbitrary divisors of $f(a)$ and $f(b)$, respectively. Doing so 
will potentially increase $q$, which will consequently lead to stronger 
restrictions on $|b|$. Even in some particular cases when $f(a)$ and 
$f(b)$ have few prime factors, to derive an effective, explicit formula 
for $q$ in the statement of Theorem \ref{thm0} is a difficult problem 
involving inequalities between products of prime powers.
However, one may obtain many corollaries of this result on the one hand 
by using some classical estimates for polynomial roots that provide
explicit upper bounds for the absolute values of the roots of $f$, and 
on the other hand by considering some special cases for the canonical 
decomposition of the two integers $f(a)$ and $f(b)$. 
                                   
The problem of finding a sharp estimate for the maximum of the absolute 
values of the roots of a given polynomial has a long history that goes 
back centuries ago. Among the earliest such attempts we mention here the 
bounds due to Cauchy and Lagrange. A generalization for Cauchy's bound 
on the largest root of a polynomial was obtained by Mignotte in \cite{MM3}:
\smallskip

{\em If a monic polynomial of height $H$ has $k$ 
roots of maximal modulus $\rho$ then $\rho < 1+H^{1/k}$.}
\smallskip

For a recent improvement of the bound of Lagrange for the maximum modulus 
of the roots we refer the reader to Batra, Mignotte, and \c Stef\u anescu \cite{BMS}. 
Further classical refinements rely on the use of some families of parameters, 
that brings considerably more flexibility, and here we only mention the 
classical methods of Fujiwara \cite{Fujiwara}, Ballieu \cite{Ballieu}, 
\cite{Marden}, Cowling and Thron \cite{Cowling1}, \cite{Cowling2}, 
Kojima \cite{Kojima}, or methods using estimates for the characteristic 
roots for complex matrices \cite{Perron}.

We will only present in this paper some simple corollaries of 
Theorem \ref{thm1}, for some cases when the canonical decompositions of 
$f(a)$ and $f(b)$ allow one to conclude that $q=1$.
\begin{corollary}\label{coro1main}
Let $f(X)=a_{0}+a_{1}X+\cdots +a_{n}X^{n}$ be a polynomial with integer 
coefficients, and $a$, $b$ two integers such that $a^2<b^2$ and
$|a_{n}|>\sum_{i=0}^{n-1}|a_{i}| \bigl(\frac{|a+b|}{2}\bigr) ^{i-n}$. 
Then $f$ is irreducible over $\mathbb{Q}$ in each of the following cases: 

\emph{i)}\ \ $|f(a)|=p^{k}r$, $|f(b)|=p^{k+1}$ with $p$ prime and 
integers $k,r$ with $k\geq 0$ and $0<r<p$;

\emph{ii)} $|f(a)|=p^{k}$, $|f(b)|=p^{k}r$ for some primes $p,r$ 
with $r<p$ and some integer $k\geq 1$.
\end{corollary}
Note that the irreducibility of $f$ will be guaranteed  solely by the
condition that $|f(b)|$ is a prime number $p$ for some integer $b$ 
with sufficiently large absolute value, without using any information 
on $a$ or on the magnitude of $p$.                                 Indeed, to conclude that $f$ is irreducible it suffices to ask $|f(b)|$ 
to be prime for some integer $b$ with $|b|>M+1$, where $M$ is the maximum 
of the absolute values of the roots of $f$. For a proof of this elementary 
fact and for some of its generalisations we refer the reader to \cite{RamMurty} 
or \cite{Girstmair}, for instance. Thus, if we ask 
$|a_{n}|>2|a_{n-1}|+2^{2}|a_{n-2}|+\cdots +2^{n}|a_{0}|$, for instance, 
then $M<\frac{1}{2}$, so if $|f(b)|$ is prime for an integer $b$ 
with $|b|\geq 2$, then $f$ must be irreducible. However, we may improve this 
result by applying Theorem \ref{thm1} with $a=0$ or Corollary \ref{coro1main} i) 
with $a=k=0$, to also include the cases $b=\pm 1$. This will seemingly come 
at the cost of asking $|f(b)|$ to exceed $|a_0|$, but as we shall see in the 
proof of the following corollary, this apparently additional condition will 
actually be an immediate consequence of our assumption on the magnitude of $|a_n|$.
\begin{corollary}
\label{coro2} Let $f(X)=a_{0}+a_{1}X+\cdots +a_{n}X^{n}$ be a polynomial with 
integer coefficients, with $|a_{n}|>2|a_{n-1}|+2^{2}|a_{n-2}|+\cdots +2^{n}|a_{0}|$ 
and $a_{0}\neq 0$. If $f(\mathbb{Z}\setminus \{0\})$ or $-f(\mathbb{Z}\setminus \{0\})$ 
contains a prime number, then $f$ must be irreducible over $\mathbb{Q}$.
\end{corollary}
We mention that Theorem 6 is a special case of Corollary \ref{coro2}, obtained 
by asking $\tilde{f}(1)$ to be prime, with $\tilde{f}$ the reciprocal of $f$, 
and asking $|a_0|>2|a_1|+2^2|a_2|+\cdots +2^n|a_n|$ instead of 
$|a_{n}|>2|a_{n-1}|+2^{2}|a_{n-2}|+\cdots +2^{n}|a_{0}|$.

Using the well-known Enestr\"om--Kakeya Theorem \cite{Kakeya}, saying 
that all the roots of a polynomial $f(X)=a_{0}+a_{1}X+\cdots +a_{n}X^{n}$ 
with real coefficients satisfying $0\leq a_{0}\leq a_{1}\leq \dots \leq a_{n}$ 
must have absolute values at most $1$, one can also prove the following two results.
\begin{corollary}\label{EK}
Let $f(X)=a_{0}+a_{1}X+\cdots +a_{n}X^{n}$ be an Enestr\"om--Kakeya polynomial 
of degree $n$ with integer coefficients, $a_0\neq 0$, and $a$, $b$ 
two integers with $a^2<b^2$ and $|a+b|>2$. Then $f$ is irreducible over 
$\mathbb{Q}$ in each of the following cases: 

\emph{i)}\ \ $|f(a)|=p^{k}r$, $|f(b)|=p^{k+1}$ with $p$ prime and 
integers $k,r$ with $k\geq 0$ and $0<r<p$;

\emph{ii)} $|f(a)|=p^{k}$, $|f(b)|=p^{k}r$ for some primes 
$p$, $r$ with $r<p$ and some integer $k\geq 1$.
\end{corollary} 
\begin{corollary}\label{EK2}
Let $f(X)=a_{0}+a_{1}X+\cdots +a_{n}X^{n}$ be an Enestr\"om--Kakeya 
polynomial of degree $n$ with integer coefficients, $a_0\neq 0$. If $f(-1)\neq 0$ 
and $|f(b)|$ is a prime number for some integer $b$ with $|b|\geq 2$, 
then $f$ is irreducible over $\mathbb{Q}$.
\end{corollary} 
We note here that condition $f(-1)\neq 0$ cannot be removed, since the 
reducible Enestr\"om--Kakeya polynomial $f(X)=X^3+X^2+X+1$ satisfies 
$|f(-2)|=5$ while $f(-1)=0$.

The proofs of the results stated so far will be given in  Section 2.

When we study the admissible divisors of $f(a)$ and $f(b)$, we distinguish 
the particular cases where $f(n)$ and $f'(n)$ are coprime for at least 
one integer $n\in \{ a,b\}$. Consequently, the set of admissible divisors of 
$f(n)$ reduces in these cases to the set $\mathcal{D}_u(f(n))$ of unitary divisors of $f(n)$.  In Section 3 we will state and prove the results corresponding to the case 
that both relations $\gcd(f(a),f'(a))=1$ and $\gcd(f(b),f'(b))=1$ hold, where 
$q$ will be denoted by $q_u$, to emphasize the role of unitary divisors of 
$f(a)$ and $f(b)$. One can easily state the results corresponding to the 
remaining two cases when only one of these relations holds. Thus we will 
present two results analogous to Theorem \ref{thm0} and Theorem \ref{thm1}, 
namely Theorem \ref{thm0unitary} and Theorem \ref{thm3}, where the radius of 
the related Apollonius circles ${\rm Ap}(a,b,q_u)$ potentially increases. Another 
benefit of using unitary divisors will consist in finding more cases when the 
canonical decompositions of $f(a)$ and $f(b)$ forces $q_u$ to be equal to $1$, 
as we shall see in Corollary \ref{coro3main}.

We will also prove in Section \ref{se4} similar results for multivariate 
polynomials $f(X_{1},\dots ,X_{r})$ over an arbitrary field $K$. The results 
for polynomials in $r\geq 3$ variables will be deduced from the results in 
the bivariate case, by writing $Y$ for $X_{r}$, $X$ for $X_{r-1}$, and 
$K$ for $K(X_{1},\dots ,X_{r-2})$. First, we will need the following definition, 
analogous to Definition \ref{admissible} for the bivariate case.
\begin{definition}\label{admissible2}
Let $K$ be a field, $f(X,Y)\in K[X,Y]$ and $a(X)\in K[X]$ such that 
$f(X,a(X))\neq 0$. We say that a polynomial $d(X)\in K[X]$ is an 
{\it admissible divisor} of $f(X,a(X))$ if $d(X)\mid f(X,a(X))$ and 
\begin{equation}\label{adm2var}
\gcd \left( d(X),\frac{f(X,a(X))}{d(X)}\right) \mid \gcd 
\left( f(X,a(X)),\frac{\partial f}{\partial Y}(X,a(X))\right) .
\end{equation}
We will denote by $D_{ad}(f(X,a(X))$ the set of admissible divisors of 
$f(X,a(X))$. Also, for $f(X,Y)$ and $a(X)$ as above we will denote 
\[
D_{u}(f(X,a(X)))=\{ d\in K[X]:d(X)|f(X,a(X)),\ \gcd\left( d(X),\frac{f(X,a(X))}{d(X)}\right) =1\} ,
\]
and call it the set of {\it unitary divisors} of $f(X,a(X))$. We note that 
in the particular case that 
$\gcd (f(X,a(X)),\frac{\partial f}{\partial Y}(X,a(X)))=1$, 
$D_{ad}(f(X,a(X))$ reduces to $\mathcal{D}_u(f(X,a(X)))$.
\end{definition}
With this definition, we have the following results.
\begin{theorem}\label{thm5}
Let $K$ be a field, $f(X,Y)=a_{0}(X)+a_{1}(X)Y+\cdots +a_{n}(X)Y^{n}\in K[X,Y]$, 
with $a_{0},\dots ,a_{n}\in K[X]$, $a_{0}a_{n}\neq 0$.
Assume that for two polynomials $a(X),b(X)\in K[X]$ we have 
$f(X,a(X))f(X,b(X))\neq 0$ and 
$\Delta:=\frac{1}{2}\cdot (\deg f(X,b(X))-\deg f(X,a(X)))\geq 0$, and let
\[
q=\max \{ \deg d_{2}-\deg d_{1}\leq \Delta :d_{1}\in D_{ad}(f(X,a(X))),d_{2}\in D_{ad}(f(X,b(X))) \}.
\]
If $\deg b(X)>\max \{ \deg a(X),\max \limits_{0\leq i\leq n-1 }
\frac{\deg a_{i}-\deg a_{n}}{n-i} \}+q$, then $f(X,Y)$ is irreducible over $K(X)$.
\end{theorem}
In particular, for $a(X)=0$ and $b(X)$ denoted by $g(X)$, we obtain:
\begin{corollary}\label{coro6} 
Let $K$ be a field, $f(X,Y)=a_{0}(X)+a_{1}(X)Y+\cdots +a_{n}(X)Y^{n}\in K[X,Y]$, 
with $a_{0},a_{1},\dots ,a_{n}\in K[X]$, $a_0a_n\neq 0$ and
\[
\deg a_{n} \geq\max \{ \deg a_{0},\deg a_{1},\dots ,\deg a_{n-1}\} .
\]
If for a non-constant polynomial $g(X)\in K[X]$, the polynomial $f(X,g(X))$ 
is irreducible over $K$, then $f(X,Y)$ is irreducible over $K(X)$.
\end{corollary}
Two additional irreducibility criteria that rely on the unitary divisors of 
$f(X,a(X))$ and $f(X,b(X))$ will be also proved in Section \ref{se4}. 
Our results are quite flexible, and provide irreducibility conditions for 
many cases where other irreducibility criteria fail. We will give in the 
last section of the paper a series of examples of infinite families of polynomials 
that are proved to be irreducible by using 
irreducibility criteria proved in previous sections.             
\section{The case of admissible divisors}  \label{se2}

{\it Proof of Theorem \ref{thm0}} \ Assume that $f$ factors as 
$f(X)=a_{n}(X-\theta _{1})\cdots (X-\theta _{n})$ for some complex numbers
$\theta _{1},\dots ,\theta _{n}$. Now let us assume to the contrary that 
$f$ is reducible, so there exist two polynomials $g,h\in \mathbb{Z}[X]$
with $\deg g=m\geq 1$, $\deg h=n-m\geq 1$ such that $f=g\cdot h$. Without 
loss of generality we may further assume that 
\[
g(X)=b_{m}(X-\theta _{1})\cdots (X-\theta _{m})\ \ \mbox{{\rm and}}\ \ 
h(X)=\frac{a_{n}}{b_{m}}(X-\theta _{m+1})\cdots (X-\theta _{n}), 
\]
for some divisor $b_{m}$ of $a_{n}$. 
Now, since $f(a)=g(a)h(a)\neq 0$ and 
$f'(a)=g'(a)h(a)+g(a)h'(a)$, and similarly $f(b)=g(b)h(b)$ and 
$f'(b)=g'(b)h(b)+g(b)h'(b)$, we see that $g(a)$ is a divisor 
$d_{1}$ of $f(a)$, and $g(b)$ is a divisor $d_2$ of $f(b)$
that must also satisfy the following divisibility conditions
\[
\gcd \left( d_1,\frac{f(a)}{d_1}\right) \mid  \gcd (f(a),f'(a))\ 
\mbox{\rm \ and\ }\ \gcd \left( d_2,\frac{f(b)}{d_2}\right) \mid  \gcd (f(b),f'(b)). 
\]
Therefore $d_1$ and $d_2$ are admissible divisors of $f(a)$ and $f(b)$, 
respectively. Similarly, if we denote $h(a)$ by $d_1'$ and $h(b)$ by $d_2'$, 
we see that $d_1'$ and $d_2'$ are also admissible divisors of $f(a)$ and $f(b)$, respectively. 
Next, since
\[
\frac{d_2}{d_1}\cdot \frac{d'_2}{d'_1}=\frac{f(b)}{f(a)},
\]
one of the quotients $\frac{|d_2|}{|d_1|}$ and $\frac{|d'_2|}{|d'_1|}$, 
say $\frac{|d_2|}{|d_1|}$, must be less than or equal to 
$\sqrt{\frac{|f(b)|}{|f(a)|}}$. In particular, we have
\begin{equation}\label{q}
\frac{|g(b)|}{|g(a)|}\leq q.
\end{equation}
We notice here that since $|f(b)|>|f(a)|$ and $1$ is obviously a divisor 
of $f(a)$ and $f(b)$, a possible candidate for $q$ is $1$, so 
$q\geq 1$. Next, we observe that we may write
\[
\frac{g(b)}{g(a)}=\frac{b-\theta_1}{a-\theta_1}\cdots \frac{b-\theta_m}{a-\theta_m},
\]
so in view of (\ref{q}) for at least one index $i\in \{ 1,\dots ,m\} $ 
we must have
\begin{equation}\label{radical}
\frac{|b-\theta_i|}{|a-\theta_i|}\leq q^{\frac{1}{m}}.
\end{equation}
Now, let us first assume that $q>1$ and all the roots of $f$ lie inside 
the Apollonius circle ${\rm Ap}(a,b,q)$. In particular, since $\theta _i$ lies 
inside the Apollonius circle ${\rm Ap}(a,b,q)$, it must satisfy the inequality $|b-\theta_i|>q|a-\theta_i|$.
Since $q>1$ and $m\geq 1$, we have $q\geq q^{\frac{1}{m}}$, so we deduce 
that we actually have
\[
\frac{|b-\theta_i|}{|a-\theta_i|}>q^{\frac{1}{m}},
\]
which contradicts (\ref{radical}). Therefore $f$ must be irreducible.

Next, assume that $q>1$ and that all the roots of $f$ lie inside the 
Apollonius circle ${\rm Ap}(a,b,\sqrt{q})$. In particular, we have 
$|b-\theta_i|>\sqrt{q}|a-\theta_i|$. Since $f$ has no rational roots, 
we must have $m\geq 2$, so $\sqrt{q}\geq q^{\frac{1}{m}}$, which also 
leads us to the desired contradiction 
\[
\frac{|b-\theta_i|}{|a-\theta_i|}>q^{\frac{1}{m}},
\]
thus proving the irreducibility of $f$.

Finally, let us assume that $q=1$, so in this case (\ref{radical}) reads
\[
\frac{|b-\theta_i|}{|a-\theta_i|}\leq 1,
\]
which is equivalent to
\begin{equation}\label{inegab}
(b-Re(\theta _i))^2\leq (a-Re(\theta _i))^2.
\end{equation}
It is easy to see that in order to contradict (\ref{inegab}), it is   
sufficient to ask all the roots of $f$ to lie in the half-plane $x<\frac{a+b}{2}$ 
if $a<b$, or in the half-plane $x>\frac{a+b}{2}$ if $a>b$. 
This completes the proof of the theorem.
\hfill  $\square $

\medskip

{\it Proof of Theorem \ref{thm1}} \ The proof goes as in the case of 
Theorem \ref{thm0}, and we deduce again that for at least one index 
$i\in \{ 1,\dots ,m\} $ we must have
\begin{equation}\label{radical2}
\frac{|b-\theta_i|}{|a-\theta_i|}\leq q^{\frac{1}{m}}.
\end{equation}
On the other hand, if $|b|>q|a|+(1+q)M$ we observe that 
\[
\frac{|b-\theta_i|}{|a-\theta_i|}\geq \frac{|b|-|\theta_i|}{|a|+|\theta_i|}
\geq \frac{|b|-M}{|a|+M}>q\geq q^{\frac{1}{m}}, 
\]
since $q\geq 1$. This contradicts (\ref{radical2}), so $f$ must be 
irreducible over $\mathbb{Q}$. 

In our second case, if we assume that $|b|>\sqrt{q}|a|+(1+\sqrt{q})M$ 
and $f$ has no rational roots, then $m\geq 2$, and consequently 
 \[
\frac{|b-\theta_i|}{|a-\theta_i|}\geq \frac{|b|-M}{|a|+M}>\sqrt{q}\geq q^{\frac{1}{m}}, 
\]
again a contradiction.

Finally, let us assume that $q=1$, $a^2<b^2$ and
$M<\frac{|a+b|}{2}$. 
If $b>a$, then $a+b>0$ and the conclusion follows by Theorem \ref{thm0} iii) 
since the disk $|z|\leq M$ containing all the roots of $f$ lies in the 
left half-plane $x<\frac{a+b}{2}$. Finally, if $a>b$, then $a+b<0$ and 
the disk $|z|\leq M$ lies in the right half-plane $x>\frac{a+b}{2}$, 
since $-M>\frac{a+b}{2}$.  \hfill  $\square $       

\medskip

{\it Proof of Corollary \ref{coro1main}} \ An immediate consequence of 
Rouch\'e's Theorem is that the condition 
$|a_{n}|>\sum_{i=0}^{n-1}|a_{i}|\bigl(\frac{|a+b|}{2}\bigr) ^{i-n}$ forces 
all the roots of $f$ to have absolute values less than $\frac{|a+b|}{2}$. 
Therefore $M<\frac{|a+b|}{2}$. In the first case we observe that if 
$|f(a)|=p^{k}r$ and $|f(b)|=p^{k+1}$, with $p$ prime and $0<r<p$, then any 
positive quotient $\frac{d_{2}}{d_{1}}$ with $d_{1}\mid f(a)$ and 
$d_{2}\mid f(b)$ has the form $\frac{p^{i}}{s}$ with $i$ an integer 
satisfying $-k\leq i\leq k+1$, and $s$ a divisor of $r$. For $i\leq 0$ 
these quotients will be at most $1$, while for $i>0$ all the corresponding 
quotients will exceed $\sqrt{\frac{p}{r}}$, as $\frac{p}{s}>\sqrt{\frac{p}{r}}$. 
Thus $q=1$ in this first case. 

Next, if $|f(a)|=p^{k}$ and $|f(b)|=p^{k}r$, any positive quotient 
$\frac{d_{2}}{d_{1}}$ with $d_{1}\mid f(a)$ and $d_{2}\mid f(b)$ 
has the form $p^{i}r^{\varepsilon }$ with $i$ an integer satisfying 
$-k\leq i\leq k$ and $\varepsilon \in \{ 0,1\} $. For $\varepsilon =0$, 
no such quotient other than 1 belongs to the interval [$1,\sqrt{r}$], 
since $p>\sqrt{r}$. Finally, we observe that for $\varepsilon =1$ no 
integer $i$ can satisfy the condition  $1< p^{i}r< \sqrt{r}$ since $p>r$. 

So in both cases $q$ must be equal to $1$. The conclusion now follows 
from Theorem \ref{thm1}.  \hfill   $\square $             

\medskip

{\it Proof of Corollary \ref{coro2}} \ 
Our assumption on the magnitude of $|a_n|$ forces all the roots of $f$ 
to have absolute values less than $\frac{1}{2}$, so $M<\frac{1}{2}$.  
We may now apply Theorem \ref{thm1} with $a=0$ or Corollary \ref{coro1main} i) 
with $a=k=0$ to deduce that $f$ is irreducible over $\mathbb{Q}$ if $b\neq 0$. 
All that remains now is to prove that our condition 
$|a_{n}|>2|a_{n-1}|+2^{2}|a_{n-2}|+\cdots +2^{n}|a_{0}|$ together with 
the fact that $|b|\geq 1$ also force the prime number $|f(b)|$ to exceed 
$|a_0|$. Indeed, we successively deduce that
\begin{eqnarray*}
|f(b)| & = & |a_0+a_1b+\cdots +a_nb^n|\geq |b|^n|a_n|-|b|^{n-1}|a_{n-1}|-\cdots -|b|\cdot |a_1|-|a_0|\\
  & > & |b|^{n}(2|a_{n-1}|+2^2|a_{n-2}|+\cdots +2^n
|a_0|)-|b|^{n-1}|a_{n-1}|-\cdots -|b|\cdot |a_1|-|a_0|\\
  & = & |b|^{n-1}(2|b|-1)|a_{n-1}|+|b|^{n-2}(2^2|b|^2-1)|a_{n-2}|+\cdots +(2^n|b|^n-1)|a_0|\\
 & \geq & (2^n|b|^n-1)|a_0|\geq (2^n-1)|a_0|\geq |a_0|,
\end{eqnarray*}
and this completes the proof.  \hfill $\square $               
\medskip

{\it Proof of Corollary \ref{EK}} \ Here, by the Enestr\"om--Kakeya 
Theorem all the roots of $f$ must have modulus at most $1$, so $M\leq 1$. 
Arguing as in the proof of Corollary \ref{coro1main} one may prove that 
$q=1$ in both cases, and the proof finishes by applying Theorem \ref{thm1}.  \hfill $\square $                
\medskip

{\it Proof of Corollary \ref{EK2}} \ We note here that since an 
Enestr\"om--Kakeya polynomial $f$ has all the roots of modulus at most $1$, 
it will be irreducible over $\mathbb{Q}$ if $|f(b)|$ is a prime for some 
integer $b$ with $|b|\geq 3$. If we consider now an additional integer 
argument $a$ and ask $f(a)\neq 0$ and $|f(b)|=p$ for some prime number 
$p>|f(a)|$, this will force $q$ to be equal to $1$, and will guarantee 
the irreducibility of $f$ via Theorem \ref{thm1} if $a^2<b^2$ and 
$|a+b|>2$. This will also allow us to use the pairs $(a,b)=(1,2)$ and 
$(a,b)=(-1,-2)$. In the first case, if $f(2)$ is a prime number, then 
it will obviously exceed $f(1)$, since $f$ has positive coefficients. 
Let us consider the remaining case $(a,b)=(-1,-2)$. If $n$ is even, 
one can easily check that $f(-2)>f(-1)\geq 0$, so if we ask $f(-1)\neq 0$, 
then condition $|f(-2)|>|f(-1)|>0$ will be obviously satisfied. On the 
other hand, if $n$ is odd, one can check that $f(-2)<f(-1)\leq 0$, so      
if $f(-1)\neq 0$, the condition $|f(-2)|>|f(-1)|>0$ will be again 
satisfied. By Theorem \ref{thm1}, $f$ will be irreducible in both cases. \hfill $\square $  
\medskip

\section{The case of unitary divisors}  \label{se3}
The aim of this section is to find irreducibility conditions by studying 
the unitary divisors of $f(a)$ and $f(b)$. Here instead of $q$ given by 
(\ref{primulq}), we will use a potentially smaller rational number, defined by
\begin{equation}\label{aldoileaq}
q_{u}=\max \left\{ \frac{d_2}{d_1}\leq \sqrt{\frac{|f(b)|}{|f(a)|}}: 
d_1\in \mathcal{D}_u(f(a)), \ d_2\in \mathcal{D}_u(f(b))\right\} .
\end{equation}
With this notation we have the  following irreducibility criterion.
\begin{theorem}\label{thm0unitary}
Let $f(X)=a_{0}+a_{1}X+\cdots +a_{n}X^{n}\in \mathbb{Z}[X]$, and assume 
that for two integers $a,b$ we have $0<|f(a)|<|f(b)|$ and 
$\gcd(f(a),f'(a))=\gcd(f(b),f'(b))=1$. Let also $q_u$ be given by {\em (\ref{aldoileaq})}.

\emph{i)} If $q_{u}>1$ and all the roots of $f$ lie inside the  
Apollonius circle ${\rm Ap}(a,b,q_u)$, then $f$ is irreducible over $\mathbb{Q}$.

\emph{ii)} If $q_{u}>1$, all the roots of $f$ lie inside the Apollonius 
circle ${\rm Ap}(a,b,\sqrt{q_u})$ and $f$ has no rational roots, then $f$ 
is irreducible over $\mathbb{Q}$.

\emph{iii)} Assume that $q_{u}=1$. If $b>a$ and all the roots of $f$ lie 
in the half-plane $x<\frac{a+b}{2}$, or if $a>b$ and all the roots of $f$ 
lie in the half-plane $x>\frac{a+b}{2}$, then $f$ is irreducible over $\mathbb{Q}$.
\end{theorem}

\begin{proof}\ Using the same notations as in the proof of Theorem \ref{thm0}, 
we see that conditions $\gcd(f(a),f'(a))=\gcd(f(b),f'(b))=1$ together 
with the divisibility conditions 
\[
\gcd \left( d_1,\frac{f(a)}{d_1}\right) \mid  \gcd (f(a),f'(a))\ 
\mbox{\rm and}\ \gcd \left( d_2,\frac{f(b)}{d_2}\right) \mid  \gcd (f(b),f'(b))
\]
will force $d_1$ to be a unitary divisor of $f(a)$, and $d_2$ to be 
a unitary divisor of $f(b)$. Similarly, $h(a)$ must be a unitary divisor 
$d_1'$ of $f(a)$ and $h(b)$ must be a unitary divisor $d_2'$ of $f(b)$. Since
\[
\frac{d_2}{d_1}\cdot \frac{d'_2}{d'_1}=\frac{f(b)}{f(a)},
\]
one of the quotients $\frac{|d_2|}{|d_1|}$ and $\frac{|d'_2|}{|d'_1|}$, 
say $\frac{|d_2|}{|d_1|}$, must be less than or equal to 
$\sqrt{\frac{|f(b)|}{|f(a)|}}$. In particular, instead of (\ref{q}), 
we obtain $\frac{|g(b)|}{|g(a)|}\leq q_u$. We note that  we will still 
have $q_{u}\geq 1$, since $1$ belongs to both $\mathcal{D}_u(f(a))$ and 
$\mathcal{D}_u(f(b))$. The proof continues as in the case of 
Theorem \ref{thm0}, with $q_u$ instead of $q$.
\end{proof}
\begin{theorem}\label{thm3} 
Let $f(X)=a_{0}+a_{1}X+\cdots +a_{n}X^{n}$ be a polynomial with integer 
coefficients, $M$ the maximum of the absolute values of its roots,
and assume that for two integers $a,b$ we have $0<|f(a)|<|f(b)|$ 
and $\gcd(f(a),f'(a))=\gcd(f(b),f'(b))=1$. Let also $q_u$ be given by relation {\em (\ref{aldoileaq})}.

\emph{i)} \ \thinspace \thinspace If $|b|>q_{u}|a|+(1+q_{u})M$, 
then $f$ is irreducible over $\mathbb{Q}$. 

\emph{ii)} \ If $|b|>\sqrt{q_{u}}|a|+(1+\sqrt{q_{u}})M$ and $f$ has 
no rational roots, then $f$ is irreducible over $\mathbb{Q}$. 

\emph{iii)} If $q_u=1$, $a^2<b^2$ and $M<\frac{|a+b|}{2}$, then 
$f$ is irreducible over $\mathbb{Q}$.
\end{theorem}
\begin{proof} \ The 
proof is similar to that of Theorem \ref{thm1}, with $q_u$ instead of $q$.
\end{proof}
In particular, we obtain the following irreducibility criterion that
complements Corollary \ref{coro1main}, by allowing one to consider 
only the unitary divisors of $f(a)$ and $f(b)$.
\begin{corollary} \label{coro3main} 
Let $f(X)=a_{0}+a_{1}X+\cdots +a_{n}X^{n}$ be a polynomial with 
integer coefficients, and $a$, $b$ two integers such that $a^2<b^2$ and
$|a_{n}|>\sum_{i=0}^{n-1}|a_{i}|(\frac{|a+b|}{2}) ^{i-n}$. Then $f$ is 
irreducible over $\mathbb{Q}$ in each of the following four cases: 

\emph{i)} $|f(a)|=p^{k_{1}}r$, $|f(b)|=p^{k_{2}}$ for some prime 
number $p$ and some  integers $k_{1},k_{2},r$ with $0\leq k_{1}<k_{2}$, 
$0<r<p$, $p\nmid f'(a)f'(b)$ and $r\nmid f'(a)$;

\emph{ii)} $|f(a)|=p^{k}$, $|f(b)|=p^{k}r^{j}$ for two distinct 
prime numbers $p$, $r$ and  some positive integers $k,j$ with 
$p^{k}>r^{j}$, $p\nmid f'(a)f'(b)$ and $r\nmid f'(b)$.

\emph{iii)} $|f(a)|=p^{u}$, $|f(b)|=q^{v}r^{t}$ for three distinct 
prime numbers $p,q,r$ and  some positive integers $u,v,t$ with 
$p^{u}>q^{v}$, $p^{u}>r^{t}$, $p^{u}<q^{v}r^{t}$, 
$p\nmid f'(a)$, $q\nmid f'(b)$ and $r\nmid f'(b)$.

\emph{iv)} $|f(a)|=p^{u}q^{v}$, $|f(b)|=r^{k}s^{l}$ for four 
distinct prime numbers $p,q,r,s$ and some positive integers $u,v,k,l$ 
with $p^{u}>r^{k}>s^{l}>q^{v}$, $p^{u}q^{v}<r^{k}s^{l}$, 
$p\nmid f'(a)$, $q\nmid f'(a)$, $r\nmid f'(b)$ and $s\nmid f'(b)$.
\end{corollary}

\begin{proof} \ Here, as in the proof of 
Corollary \ref{coro1main} we have $M<\frac{|a+b|}{2} $. 

Assuming now that $|f(a)|=p^{k_{1}}r$ and $|f(b)|=p^{k_{2}}$ with 
$0<r<p$ and $k_{1}<k_{2}$, then any $d_1\in\mathcal{D}_u(f(a))$ 
is of the form $s,p^{k_{1}}$ or $p^{k_1}s$, with $s\in\mathcal{D}_u(r)$, 
while a unitary divisor $d_2$ of $f(b)$ is either $1$ or $p^{k_{2}}$. 
Therefore $\frac{d_2}{d_1}$ is either $\frac{1}{d_1}$, which is at 
most $1$, or is of the form $\frac{p^{k_2}}{s}$, $p^{k_2-k_1}$,  
$\frac{p^{k_2-k_1}}{s}$ with $s\in\mathcal{D}_u(r)$. It suffices   
to observe that for each $s\in\mathcal{D}_u(r)$ we have 
$\frac{p^{k_2-k_1}}{s}>\sqrt{\frac{p^{k_2}}{p^{k_1}r}}$, as 
$p^{k_2-k_1}>\frac {s^2}{r}$. Thus $q_{u}=1$ in this first case.

For our second case let us assume that $|f(a)|=p^{k}$ and $|f(b)|=p^{k}r^{j}$ 
for two distinct primes $p$, $r$ and some positive integers $k,j$ with 
$p^{k}>r^{j}$. Then $\mathcal{D}_u(f(a))=\{ 1,p^{k}\} $ and 
$\mathcal{D}_u(f(b))=\{ 1,p^{k},r^{j},p^{k}r^{j}\} $.  Therefore any 
quotient $\frac{d_{2}}{d_{1}}$ with $d_{1}\in \mathcal{D}_u(f(a))$ 
and $d_{2}\in \mathcal{D}_u(f(b))$ belongs to the set 
$\{ 1,p^{k},r^{j},p^{k}r^{j},\frac{1}{p^{k}},\frac{r^{j}}{p^{k}}\}$, 
so here again it holds                 
$q_{u}=1$, since according to our assumption that $p^{k}>r^{j}$, the 
only such quotient in the interval [$1,r^{\frac{j}{2}}$] is $1$.

In our third case we have $\mathcal{D}_u(f(a))=\{ 1,p^{u}\} $ and 
$\mathcal{D}_u(f(b))=\{ 1,q^{v},r^{t},q^{v}r^{t}\} $, so any quotient 
$\frac{d_{2}}{d_{1}}$ with $d_{1}\in \mathcal{D}_u(f(a))$ and $d_{2}=1$ 
will be at most $1$, while any quotient $\frac{d_{2}}{d_{1}}$ with 
$d_{1}\in \mathcal{D}_u(f(a))$ and $d_{2}=q^{v}r^{t}$ will exceed 
$\sqrt{\frac{q^{v}r^{t}}{p^{u}}}$. We are thus left with the case that 
\[
\frac{d_2}{d_1}\in \left\{ q^v,\frac{q^v}{p^u},r^{t},\frac{r^t}{p^u}\right\} .
\]
It is now plain to see that while $\frac{q^v}{p^u}$ and $\frac{r^t}{p^u}$ 
are less than $1$, both $q^{v}$ and $r^{t}$ exceed 
$\sqrt{\frac{q^{v}r^{t}}{p^{u}}}$, so here  we have $q_{u}=1$ too.  

In our last case we have $\mathcal{D}_u(f(a))=\{ 1,p^{u},q^{v}, p^{u}q^{v}\} $ 
and $\mathcal{D}_u(f(b))=\{ 1,r^{k},s^{l},r^{k}s^{l}\} $. Here we first 
note that any quotient $\frac{d_{2}}{d_{1}}$ with $d_{1}\in \mathcal{D}_u(f(a))$ 
and $d_{2}=1$ will be at most $1$, and any quotient $\frac{d_{2}}{d_{1}}$ 
with $d_{1}\in \mathcal{D}_u(f(a))$ and $d_{2}=r^{k}s^{l}$ will exceed 
$\sqrt{\frac{r^{k}s^{l}}{p^{u}q^{v}}}$. We are therefore left with the case that 
\[
\frac{d_2}{d_1}\in \left\{ r^k,\frac{r^k}{p^u},\frac{r^k}{q^v},
\frac{r^k}{p^uq^v},s^l,\frac{s^l}{p^u},\frac{s^l}{q^v},\frac{s^l}{p^uq^v}\right\} .
\]
Using now our hypothesys that $p^u>r^k>s^l>q^v$ it is easy to check that each 
of the quotients $\frac{r^k}{p^u},\frac{r^k}{p^uq^v},\frac{s^l}{p^u},\frac{s^l}{p^uq^v}$ 
is less than $1$, while each of the remaining ones 
$r^k,\frac{r^k}{q^v},s^l,\frac{s^l}{q^v}$ exceeds 
$\sqrt{\frac{r^{k}s^{l}}{p^{u}q^{v}}}$, so here $q_u=1$ as well.
The irreducibility of $f$ now follows from Theorem \ref{thm3}. 
\end{proof}
The reader may naturally wonder if there exists a result analogous to 
Corollary \ref{coro2}, that uses information on the prime power values 
of a polynomial, instead of its prime values. The answer is affirmative, and one can      
prove the following result, which illustrates a situation when there is 
no need to impose both conditions $\gcd(f(a),f'(a))=1$ and $\gcd(f(b),f'(b))=1$.
\begin{corollary}
\label{corovechi} Let $f(X)=a_{0}+a_{1}X+\cdots +a_{n}X^{n}\in \mathbb{Z}[X]$ 
with $a_0a_n\neq 0$ and $|a_{n}|>2|a_{n-1}|+2^{2}|a_{n-2}|+\cdots +2^{n}|a_{0}|$. 
If for a non-zero integer $m$, a prime number $p$ and an integer $k\geq 2$ we 
have $|f(m)|=p^{k}$ and $p\nmid f'(m)$, then $f$ is irreducible over $\mathbb{Q}$.
\end{corollary}
\begin{proof} \ Assume first that for a polynomial $f$ with integer coefficients 
and two integers $a$, $b$ we have $0<|f(a)|<|f(b)|=p^{k}$ for some prime number 
$p$ and some positive integer $k$. Let us also assume that $p\nmid f'(b)$. 
Then $\mathcal{D}_{ad}(f(b))=\mathcal{D}_{u}(f(b))$ and any positive quotient 
$\frac{d_{2}}{d_{1}}$ in the definition of $q$ in Theorem \ref{thm1} either 
has the form $\frac{1}{d_{1}}$ with $d_{1}$ an admissible divisor of $f(a)$, 
and hence is at most $1$, or is equal to $\frac{p^{k}}{d_{1}}$, which obviously 
exceeds $\sqrt{\frac{p^{k}}{|f(a)|}}$. Therefore in this case $q$ must be equal 
to $1$. In particular, if we take $a=0$ and $b=m$, and assume that 
$0<|a_0|<|f(m)|=p^{k}$ for some prime number $p$ and some integer $k\geq 2$, 
and also assume that $p\nmid f'(m)$, then the corresponding $q$ must be equal 
to $1$. Since all the roots of $f$ have absolute values less than $\frac{1}{2}$,
we have $M<\frac{1}{2}$. By Theorem \ref{thm1} we conclude that $f$ is irreducible 
over $\mathbb{Q}$ if $|m|>2M$, or equivalently, if $|m|\geq 1$, since 
$M<\frac{1}{2}$. All that remains is to prove that the inequalities $|m|\geq 1$ 
and $|a_{n}|>2|a_{n-1}|+2^{2}|a_{n-2}|+\cdots +2^{n}|a_{0}|$ actually force 
$|f(m)|$ to exceed $|a_0|$. As in the case of Corollary \ref{coro2}, we deduce 
successively that
\begin{eqnarray*}
|f(m)| & = & |a_0+a_1m+\cdots +a_nm^n|\geq |m|^n |a_n|-|m|^{n-1}|a_{n-1}|-
\cdots -|m|\cdot |a_1|-|a_0|\\  
& > & |m|^{n}(2|a_{n-1}|+2^2|a_{n-2}|+
\cdots +2^n|a_0|)-|m|^{n-1}|a_{n-1}|-\cdots -|m|\cdot |a_1|-|a_0|\\ 
& = & |m|^{n-1}(2|m|-1)|a_{n-1}|+|m|^{n-2}(2^2|m|^2-1)|a_{n-2}|+\cdots +(2^n|m|^n-1)|a_0|\\
& \geq & (2^n|m|^n-1)|a_0|\geq (2^n-1)|a_0|\geq |a_0|,
\end{eqnarray*}
completing the proof. We note that using Theorem \ref{thm3} instead of 
Theorem \ref{thm1} would impose here the unnecessary aditional condition 
$\gcd(f(0),f'(0))=1$, that is $\gcd(a_0,a_1)=1$.
\end{proof}
We mention here that Theorem 7 is a special case of Corollary \ref{corovechi}, 
obtained by considering the reciprocal of $f$ instead of $f$.

\section{The case of multivariate polynomials}  \label{se4}
{\it Proof of Theorem \ref{thm5}} \ We will first introduce a nonarchimedean
absolute value $|\cdot|$ on $K(X)$, as follows. We first fix an arbitrary real number
$\rho>1$, and for any polynomial $F(X)\in K[X]$ we define $|F(X)|$ by the equality
\[
|F(X)|=\rho^{\deg F(X)}.
\]
We then extend this absolute value $|\cdot|$ to $K(X)$ by multiplicativity,
that is, for any polynomials $F(X),G(X)\in K[X]$, $G(X)\neq 0$, we let   
$\left| \frac{F(X)}{G(X)} \right|=\frac{|F(X)|}{|G(X)|}$.  
Here we must note that for any non-zero element $F$ of $K[X]$ one has $|F|\geq 1$.

Let now $\overline{K(X)}$ be a fixed algebraic closure of $K(X)$, and let us fix an
extension of our absolute value $|\cdot|$ to $\overline{K(X)}$, which we will
also denote by $|\cdot|$.

Suppose $f$ as a polynomial in $Y$ with coefficients in $K[X]$ factorizes as 
\[
f(X,Y)=a_{n}(X)(Y-\theta _{1})\cdots (Y-\theta _{n})
\]
for some $\theta _{1},\dots ,\theta _{n}\in \overline{K(X)}$. 

Next, we will prove that 

\begin{equation}\label{grade}
\max \{ |\theta _{1}|,\dots ,|\theta _{n}|  \} \leq \rho ^ {\ \max 
\limits_{0\leq i\leq n-1 }\frac{\deg a_{i}-\deg a_{n}}{n-i}}.
\end{equation}

To prove this claim, let $\lambda :=\max \limits_{0\leq i\leq n-1 }
\frac{\deg a_{i}-\deg a_{n}}{n-i}$, and let us assume to the contrary 
that $f$ has a root $\theta $ with $|\theta |> \rho ^{\lambda }$. Since 
$\theta \ne 0$ and our absolute value also satisfies the triangle 
inequality, we successively deduce that
\begin{eqnarray*}
0=\left| \sum\limits _{i=0}^{n}a_{i}\theta ^{i-n}\right| & \geq & 
|a_{n}|- \left| \sum\limits _{i=0}^{n-1}a_{i}\theta ^{i-n}\right|
\geq |a_{n}|- \max\limits _{0\leq i\leq n-1}|a_{i}|\cdot |\theta |^{i-n}\\
& > & |a_{n}|- \max\limits _{0\leq i\leq n-1}|a_{i}|\cdot \rho^{(i-n)\lambda },
\end{eqnarray*}
yielding $|a_{n}|< \max\limits _{0\leq i\leq n-1}|a_{i}|\cdot \rho^{(i-n)\lambda }$, 
or equivalently
\begin{equation}\label{gradenou}
\deg a_{n}< \max\limits _{0\leq i\leq n-1}\{ \deg a_{i}+(i-n)\lambda \} .
\end{equation}
Let us select now an index $k\in \{ 0,\dots , n-1\} $ for which the maximum 
in the right side of (\ref{gradenou}) is attained. Then we deduce that
\[
\deg a_{n}< \deg a_{k}+(k-n)\lambda ,
\]
which leads us to
\[
\frac{\deg a_{k}-\deg a_{n}}{n-k}>\max \limits_{0\leq i\leq n-1 }\frac{\deg a_{i}-\deg a_{n}}{n-i},
\]
a contradiction. Therefore (\ref{grade}) holds, so 
$|\theta _{i}|\leq\rho ^{\lambda }$ for $i=1,\dots ,n$.

Now let us assume to the contrary that $f$ is reducible, so by the celebrated 
Gauss' Lemma there exist two polynomials $g,h\in K[X,Y]$
with $\deg _{Y} g=m\geq 1$, $\deg _{Y} h=n-m\geq 1$ such that $f=g\cdot h$. 
Without loss of generality we may further assume that 
\[
g(X,Y)=b_{m}(X)(Y-\theta _{1})\cdots (Y-\theta _{m})\ \ \mbox{{\rm and}}\ \ 
h(X,Y)=\frac{a_{n}(X)}{b_{m}(X)}(Y-\theta _{m+1})\cdots (Y-\theta _{n}), 
\]
for some divisor $b_{m}(X)$ of $a_{n}(X)$. Since we have $f(X,a(X))=g(X,a(X))h(X,a(X))$ 
and $f(X,b(X))=g(X,b(X))h(X,b(X))$, and also
\begin{eqnarray*}
\frac{\partial{f}}{\partial{Y}}(X,a(X)) & = & 
\frac{\partial{g}}{\partial{Y}}(X,a(X))h(X,a(X))+g(X,a(X))\frac{\partial{h}}{\partial{Y}}(X,a(X)),\\
\frac{\partial{f}}{\partial{Y}}(X,b(X)) & = & 
\frac{\partial{g}}{\partial{Y}}(X,b(X))h(X,b(X))+g(X,b(X))\frac{\partial{h}}{\partial{Y}}(X,b(X)),
\end{eqnarray*}
we see that $g(X,a(X))$ is a divisor $d_{1}$ of $f(X,a(X))$, and $g(X,b(X))$ is a divisor 
$d_2$ of $f(X,b(X))$
that must also satisfy the following divisibility conditions
\begin{eqnarray*}
\gcd \left( d_1,\frac{f(X,a(X))}{d_1}\right) & \mid & \gcd 
\left( f(X,a(X)),\frac{\partial f}{\partial Y}(X,a(X))\right) \ \mbox{\rm and}\\ 
\gcd \left( d_2,\frac{f(X,b(X))}{d_2}\right) & \mid & \gcd 
\left( f(X,b(X)),\frac{\partial f}{\partial Y}(X,b(X))\right) . 
\end{eqnarray*}
Recalling condition (\ref{adm2var}), we see that $g(X,a(X))$ is actually 
an admissible divisor $d_{1}$ of $f(X,a(X))$, and $g(X,b(X))$ is an 
admissible divisor $d_2$ of $f(X,b(X))$. Similarly, $h(X,a(X))$ is an 
admissible divisor $d_{1}'$ of $f(X,a(X))$, and $h(X,b(X))$ is an 
admissible divisor $d_2'$ of $f(X,b(X))$. Therefore we have
\[
\frac{g(X,b(X))}{g(X,a(X))}=\frac{d_2(X)}{d_1(X)}\ \mbox{\rm and}\ 
\frac{h(X,b(X))}{h(X,a(X))}=\frac{d'_2(X)}{d'_1(X)} 
\]
with $d_1,d'_{1}\in D_{ad}(f(X,a(X)))$ and $d_2,d'_2\in D_{ad}(f(X,b(X)))$. Now, since
\[
\frac{d_2(X)}{d_1(X)}\cdot \frac{d'_2(X)}{d'_1(X)}=\frac{f(X,b(X))}{f(X,a(X))},
\]
by applying $|\cdot |$, we see that one of the quotients $\frac{|d_2(X)|}{|d_1(X)|}$ 
and $\frac{|d'_2(X)|}{|d'_1(X)|}$, say $\frac{|d_2(X)|}{|d_1(X)|}$, must be 
less than or equal to 
$\sqrt{\frac{|f(X,b(X))|}{|f(X,a(X))|}}$. In particular, we have
\begin{equation}\label{qmultiv}
\frac{|g(X,b(X))|}{|g(X,a(X))|}\leq \rho ^{q}.
\end{equation}
We notice now that $q\geq 0$ because $\Delta \geq 0$ and $d_{1}=1$, $d_{2}=1$ 
is obviously a pair of divisors of  $f(X,a(X))$ and $f(X,b(X))$, 
respectively, so  that $0=\deg 1-\deg 1$ is a possible candidate for 
$q$. Next, we observe that we may write
\[
\frac{g(X,b(X))}{g(X,a(X))}=\frac{b(X)-\theta_1}{a(X)-\theta_1}\cdots 
\frac{b(X)-\theta_m}{a(X)-\theta_m},
\]
so in view of (\ref{qmultiv}) for at least one index $i\in \{ 1,\dots ,m\} $ 
we must have
\begin{equation}\label{radicalmultiv}
\frac{|b(X)-\theta_i|}{|a(X)-\theta_i|}\leq \rho ^{\frac{q}{m}}.
\end{equation}
On the other hand, since our absolute value also satisfies the triangle 
inequality, we see that
\[
\frac{|b(X)-\theta_i|}{|a(X)-\theta_i|}\geq \frac{|b(X)|-|\theta_i|}{|a(X)|+|\theta_i|}
\geq \frac{|b(X)|-\rho ^{\lambda }}{|a(X)|+\rho ^{\lambda }}
=\frac{\rho ^{\deg b(X)}-\rho ^{\lambda }}{\rho ^{\deg a(X)}+\rho ^{\lambda }}.
\]
We will now prove that for a sufficiently large $\rho $ one has
\[
\frac{\rho ^{\deg b(X)}-\rho ^{\lambda }}{\rho ^{\deg a(X)}+\rho ^{\lambda }}>
\rho ^{q}\geq \rho ^{\frac{q}{m}},
\]
and this will contradict (\ref{radicalmultiv}). The inequality 
$\rho ^{q}\geq \rho ^{\frac{q}{m}}$ obviously holds for an arbitrary
$\rho >1$ since $q\geq 0$ and $m=\deg _{Y}g\geq 1$. Finally, all that 
remains to see is that the first inequality is equivalent to
\[
\rho ^{\deg b(X)}>\rho ^{q+\deg a(X)}+\rho ^{q+\lambda }+\rho ^{\lambda },
\]
which will obviously hold for a sufficiently large $\rho $, since 
according to our assumption on the magnitude of $\deg b(X)$ we have 
$\deg b(X)>\max \{ q+\deg a(X), q+\lambda ,\lambda \} $.
Therefore $f$ must be irreducible over $K(X)$,
and this completes the proof.  \hfill  $\square $

\medskip

{\it Proof of Corollary \ref{coro6}} \ We may apply Theorem \ref{thm5} 
with $a(X)=0$ and $b(X)$ any non-constant polynomial (denoted here by 
$g(X)$) such that $f(X,b(X))$ is irreducible over $K$. To see this, we 
first observe that $f(X,a(X))=f(X,0)=a_0\neq 0$ and since $f(X,b(X))$ 
is irreducible, we also have $f(X,b(X))\neq 0$. Next, the inequality 
$\deg a_n\geq \max \{ \deg a_0,\dots ,\deg a_{n-1} \} $ shows that 
\begin{equation}\label{diferenta}
\deg f(X,b(X))=n\deg b+\deg a_n>\deg a_0=\deg f(X,a(X)),
\end{equation}
so the condition $\Delta \geq 0$ is also satisfied. It remains to prove 
that in this case we have $q=0$. To prove this equality, we note that 
any divisor $d_{2}$ of the irreducible polynomial $f(X,b(X))$ either 
has degree $0$ or has degree $n\deg b+\deg a_n$, so $\deg d_2-\deg d_1$ 
in the definition of $q$ is equal either to $-\deg d_1$, which is at most 
$0$, or to $n\deg b+\deg a_n-\deg d_1$, which exceeds $\Delta $, in 
view of inequality (\ref{diferenta}). Therefore our condition 
$\deg b(X)>\max \{ \deg a(X),\max \limits_{0\leq i\leq n-1 }
\frac{\deg a_{i}-\deg a_{n}}{n-i} \}+q$
reduces in this case to $\deg b(X)>0$.  \hfill  $\square $
\medskip

We mention that in analogy to the univariate case, when testing the 
irreducibility of $f(X,Y)$ in terms of two of its values $f(X,a(X))$ 
and $f(X,b(X))$, we don't necessarily need to impose conditions on 
both partial derivatives $\frac{\partial f}{\partial Y}(X,a(X))$ and 
$\frac{\partial f}{\partial Y}(X,b(X))$. However, we will only state 
here a result for the case that $f(X,a(X))$ and 
$\frac{\partial f}{\partial Y}(X,a(X))$ are relatively prime, and 
$f(X,b(X))$ and $\frac{\partial f}{\partial Y}(X,b(X))$ are also relatively prime. 
\begin{theorem}\label{thm7}
Let $K$ be a field, $f(X,Y)=a_{0}(X)+a_{1}(X)Y+\cdots +a_{n}(X)Y^{n}\in K[X,Y]$, 
with $a_{0},\dots ,a_{n}\in K[X]$, $a_{0}a_{n}\neq 0$. Assume that for two 
polynomials $a(X),b(X)\in K[X]$ we have $f(X,a(X))f(X,b(X))\neq 0$ and 
$\Delta:=\frac{1}{2}\cdot (\deg f(X,b(X))-\deg f(X,a(X)))\geq 0$, and let
\[
q_{u}=\max \{ \deg d_{2}-\deg d_{1}\leq \Delta :d_{1}\in \mathcal{D}_u(f(X,a(X))),d_{2}\in \mathcal{D}_u(f(X,b(X))) \}.
\]
If $\gcd(f(X,a(X)), \frac{\partial f}{\partial Y}(X,a(X)))=1$,
$\gcd(f(X,b(X)), \frac{\partial f}{\partial Y}(X,b(X)))=1$ and
\[
\deg b(X)>\max \left\{ \deg a(X),\max \limits_{0\leq i\leq n-1 }\frac{\deg a_{i}-\deg a_{n}}{n-i} \right\} +q_{u},
\]
then $f(X,Y)$ is irreducible over $K(X)$.
\end{theorem}
\begin{proof} \ Here, with the same notations as in the proof of 
Theorem \ref{thm5}, we see that $d_{1}$ and $d_1'$ must belong to 
$\mathcal{D}_u(f(X,a(X)))$, while $d_{2}$ and $d_2'$ must belong 
to $\mathcal{D}_u(f(X,b(X)))$. We notice here that even if the defining set 
for $q_{u}$ is in this case smaller than the one for $q$ in Theorem \ref{thm5},
we will still have $q_{u}\geq 0$, since $1$ belongs to both 
$\mathcal{D}_u(f(X,a(X)))$ and $\mathcal{D}_u(f(X,b(X)))$. The rest of the proof 
is identical to that of Theorem \ref{thm5}, and will be omitted.
\end{proof}
In particular, we obtain as a special case the following irreducibility 
criterion that complements Corollary \ref{coro6}, by allowing $f(X,g(X))$ 
to be a power of an irreducible polynomial, instead of an irreducible polynomial.
\begin{corollary}\label{coro8} 
Let $K$ be a field, $f(X,Y)=a_{0}(X)+a_{1}(X)Y+\cdots +a_{n}(X)Y^{n}\in K[X,Y]$, 
with $a_{0},a_{1},\dots ,a_{n}\in K[X]$, $a_0a_n\neq 0$ and
\[
\deg a_{n} \geq\max \{ \deg a_{0},\deg a_{1},\dots ,\deg a_{n-1}\} .
\]
If for a non-constant polynomial $g(X)\in K[X]$, the polynomial $f(X,g(X))$ 
is a power of an irreducible polynomial over $K$, and $f(X,g(X))$ and 
$\frac{\partial f}{\partial Y}(X,g(X))$ are relatively prime, then 
$f(X,Y)$ is irreducible over $K(X)$.
\end{corollary}
\begin{proof} \ Here we may apply Theorem \ref{thm7} with $a(X)=0$, and 
$b(X)$ any non-constant polynomial (denoted here by $g(X)$) such that 
$f(X,b(X))=h(X)^{k}$ with $k\geq 1$ and $h\in K[X]$, $h$ irreducible over 
$K$. Indeed, in this case $f(X,a(X))f(X,b(X))\neq 0$, and the fact that 
$\Delta \geq 0$ follows again by (\ref{diferenta}). It remains to prove 
that we must have $q_{u}=0$. Any unitary divisor $d_{1}$ of 
$f(X,a(X))=a_{0}$ is of degree at most $\deg a_0$, and any divisor 
$d_{2}$ of $h(X)^{k}$ which is relatively prime to $\frac{h(X)^{k}}{d_{2}}$ 
either has degree $0$ or has degree $k\deg h=n\deg b+\deg a_n$, so 
$\deg d_2-\deg d_1$ in the definition of $q_u$ is  equal either to 
$-\deg d_1$, which is at most $0$, or to $n\deg b+\deg a_n-\deg d_1$, 
which exceeds $\Delta $, according to (\ref{diferenta}). Therefore our 
condition $\deg b(X)>\max \{ \deg a(X),\max \limits_{0\leq i\leq n-1 }
\frac{\deg a_{i}-\deg a_{n}}{n-i} \} $+$q_{u}$ reduces here to $\deg b(X)>0$ too.
\end{proof}

\section{Examples} \label{se5}
{\bf 1)\ } For any fixed, arbitrarily chosen integers 
$a_{1},\dots ,a_{n-1}$ and $k\geq 0$, the polynomial
\[
f(X)=p^{k}+a_{1}X+\cdots +a_{n-1}X^{n-1}+(p^{k+1}-p^{k}-a_{1}-\dots -a_{n-1})X^{n}
\]
is irreducible over $\mathbb{Q}$ for all but finitely many prime numbers $p$. 
To prove this, we note that $f(0)=p^{k}$ and $f(1)=p^{k+1}$, so we may apply 
Corollary \ref{coro1main} i) with $a=0, \ b=1$, provided that 
\[
|p^{k+1}-p^{k}-a_{1}-\dots -a_{n-1}|>2^{n}p^{k}+\sum\limits_{i=1}^{n-1}2^{n-i}|a_{i}|,
\]
and this will obviously hold for sufficiently large prime numbers $p$.

To see an explicit example where a lower bound for $p$ can be easily derived,
one can take $a_{1}=\dots =a_{n-1}=1$ and $k\geq 1$ to conclude that      
the polynomial 
\[
f(X)=(p^{k+1}-p^{k}-n+1)X^{n}+X^{n-1}+\cdots +X+p^{k}
\]
is irreducible over $\mathbb{Q}$ for all primes $p\geq 2^{n}+3$.
Indeed, the condition
\[
|a_n|=p^{k+1}-p^{k}-n+1>2^{n}-2+2^{n}p^{k}=\sum\limits_{i=0}^{n-1}2^{n-i}|a_i| 
\]
will obviously hold for $p\geq 2^{n}+3$.
 
\smallskip 

{\bf 2)\ } For any integers $k\geq 1$, $n\geq 1$,
and any prime numbers 
$p$, $r$ with $p>r\geq 2^{n}+1$, the polynomial
$f(X)=\bigl(p^{k}(r+1)+n-1\bigr)X^{n}-X^{n-1}-\cdots -X-p^{k}$ is irreducible 
over $\mathbb{Q}$. Here we observe that $|f(0)|=p^{k}$, $f(1)=p^{k}r$, and 
\[
|a_n|=p^{k}(r+1)+n-1>2^{n}-2+2^{n}p^{k}=\sum\limits_{i=0}^{n-1}2^{n-i}|a_i| 
\]
for $p>r\geq 2^{n}+1$. The conclusion follows by Corollary \ref{coro1main} ii) 
with $a=0$ and $b=1$.

\smallskip 

{\bf 3)\ } For any integers $n\geq 1$ and $m>2^{n+1}-2$ such that 
$(m+1)\cdot 2^{n}-1$ is a prime number, the polynomial 
$f(X)=1+X+\cdots +X^{n-1}+mX^{n}$ is irreducible over $\mathbb{Q}$. To see 
this, we note that $f(2)=(m+1)\cdot 2^{n}-1$, which is a prime number, and 
since the inequality $m>2^{n+1}-2$ is precisely the condition 
$|a_n|>2|a_{n-1}|+2^2|a_{n-2}|+\cdots +2^n|a_0|$ applied to the coefficients 
of $f$, the conclusion follows by Corollary \ref{coro2}.

\smallskip 

{\bf 4)\ } Let $f(X)=254X^6-4X^5+X^4-X^3-X^2-3$. We observe that 
$f(2)=127^{2}$, which is a prime power, 
and $f'(2)=48464$, which is not divisible by $127$. Then, since 
$a_6$ satisfies $|a_6|=254>228=\sum\limits_{i=0}^{5}2^{6-i}|a_i|$, 
we see from Corollary \ref{corovechi} that $f$ too is irreducible 
over $\mathbb{Q}$.   

\smallskip 

{\bf 5)\ } Let $p$ be a prime number and let 
\[
f(X,Y)=p+(p-1)XY+(p^{2}X+p+1)Y^{2}+pXY^{3}+X^{2}Y^{4}.
\]
We observe that if we write $f$ as a polynomial in $Y$ with coefficients 
in $\mathbb{Z}[X]$ as $f=\sum\limits_{i=0}^{4}a_{i}(X)Y^{i}$ with 
$a_{i}(X)\in \mathbb{Z}[X]$, we have $\deg a_{4}>\max_{0\leq i\leq 3}
\{ \deg a_{i} \} $. On the other hand, we note that
\[
f(X,X)=p+2pX^{2}+p^{2}X^{3}+pX^{4}+X^{6}, 
\]
which is Eisensteinian with respect to the prime $p$, and hence 
irreducible over $\mathbb{Q}$, so one may apply Corollary \ref{coro6} 
with $g(X)=X$ to conclude that $f$ is irreducible over $\mathbb{Q}$.

\smallskip 

{\bf 6)\ } Let $p$ be a prime number and let 
\[
f(X,Y)=p^{2}+(2p^{3}+p^{4}X)Y+(2pX+2p^{2}X^{2})Y^{2}+X^{2}Y^{4}.
\]
If we write $f$ as $f=\sum\limits_{i=0}^{4}a_{i}(X)Y^{i}$ with 
$a_{i}(X)\in \mathbb{Z}[X]$, we have $\deg a_{4}\geq\max_{0\leq i\leq 3}
\{ \deg a_{i} \} $. We observe next that
$f(X,X)=(p+p^{2}X+X^{3})^{2}$,
with $p+p^{2}X+X^{3}$ irreducible over $\mathbb{Q}$, being Eisensteinian 
with respect to $p$. Since
\[
\frac{\partial f}{\partial Y}(X,X)=2p^{3}+p^{4}X+4pX^{2}+4p^{2}X^{3}+4X^{5},
\]
we have
\[
\frac{\partial f}{\partial Y}(X,X)=4X^{2}(p+p^{2}X+X^{3})+2p^{3}+p^{4}X,
\]
so $\frac{\partial f}{\partial Y}(X,X)$ is not divisible by $p+p^{2}X+X^{3}$, 
which shows that $f(X,X)$ and $\frac{\partial f}{\partial Y}(X,X)$ 
are
relatively prime. We may therefore apply 
Corollary \ref{coro8} with $g(X)=X$ to conclude that $f$ is irreducible 
over $\mathbb{Q}$.
\medskip

{\bf Acknowledgements} \ This work was partially done in the frame of the 
GDRI ECO-Math.  The authors are grateful to C.M. Bonciocat for useful discussions and suggestions that improved the presentation of the paper.

\end{document}